\documentclass[a4paper,12pt]{amsart}
  \usepackage{amssymb,amsthm}
\usepackage{graphicx,color}        
 
\usepackage{url} 

\newtheorem{theorem}{Theorem}[section]

\newtheorem{proposition}[theorem]{Proposition}

\theoremstyle{definition}
\newtheorem{definition}[theorem]{Definition}

\newtheorem{remark}[theorem]{Remark}

\def\r{\mathbb R}
\def\s{\mathbb S}
 \def\l{\mathbb L}
  
 \def\n{\mathbf n}

\begin{document}

\title{A characterization of rotational minimal surfaces  in the de Sitter space}
\author{Rafael L\'opez}
\address{ Departamento de Geometr\'{\i}a y Topolog\'{\i}a\\  Universidad de Granada.  18071 Granada, Spain}
\email{rcamino@ugr.es}
\subjclass{53C40,  53C42}
\keywords{de Sitter space, spacelike surface, timelike surface, minimal surface, catenary}
%
%

\begin{abstract} The generating curves of rotational minimal surfaces in   the de Sitter space $\s_1^3$ are characterized as solutions of a variational problem. It is proved that these curves are the critical points of the center of mass among all curves of $\s_1^2$ with prescribed endpoints and fixed length.  This extends the known properties of the catenary and the catenoid in the Euclidean setting. 
\end{abstract}

\maketitle

\section{Introduction and motivation}\label{s1}

 A non-degenerate surface in $3$-dimensional de Sitter space $\s_1^3$ is called a {\it minimal surface} if the mean curvature $H$ vanishes on the surface.  Although the equation $H=0$ is equivalent that the trace of the second fundamental form vanishes on the surface,  there is no a    variational interpretation of these surfaces such as it occurs in  Euclidean space. For spacelike surfaces of $\s_1^3$, the induced metric is Riemannian and it makes sense to define the area of a surface. In this case, a   surface of $\s_1^3$ with zero constant mean curvature is a   local maximizer of the area functional. If the surface is timelike, the induced metric has index one and there is not a notion of the area of a timelike surface.  Despite these differences with the Euclidean case, and in order to simplify the statements our results,  the terminology of minimal surfaces  to name a non-degenerate surface of $\s_1^3$ with zero constant mean curvature   is kept in this article.  

An interesting class of surfaces of $\s_1^3$ are the surfaces of revolution, also called, rotational surfaces. The family of surfaces of revolution in $\s_1^3$  is richer than the Euclidean setting because the rotational axis can be of three different causal character.  If, in addition, we investigate which of these surfaces are minimal, we need to distinguish  if the surface is spacelike and timelike. The  property that the surface is rotational implies that the equation $H=0$  is an ordinary differential equation given  in terms of the generating curve of the surface. This makes that the classification of rotational minimal surfaces of $\s_1^3$ can be obtained:  \cite{ak,li1,ll1,mo}; see also  \cite{br,li2,ya}. Recently,   rotational surfaces in $\s_1^3$ with a Weingarten relation between its curvatures have been investigated  because the Weingarten relation is described again  in terms of   ordinary differential equations \cite{de,ll2}.

If the classification of rotational minimal surfaces of $\s_1^3$ is deduced by solving ordinary differential equations, there is not known a special property of the generating curves. This contrasts to the Euclidean situation which it is  recalled giving a brief historical account.  Euler proved in 1744 that the catenoid is the only non-planar rotational minimal surface in  $\r^3$ being the catenary its generating curve.  This curve is also  the solution of the question that asks by the shape of a chain hanging by its own weight. This problem  attracted the interest of many scientistics some centuries above, beginning with Galileo. The answer, the catenary, was obtained some decades before to the Euler's result joining efforts of Hooke, Leibniz, Huygens and Bernouilli, among others. 

Coming back to the de Sitter space, and motivated by the property  that  the generating curve of the catenoid is variationally  characterized as the shape of a hanging chain, the purpose of this paper is to answer the following

\begin{quote}
{\it Question:} Can the generating curves of rotational minimal surfaces of  $\s_1^3$ be characterized from a variational viewpoint?
\end{quote}

   In $\r^2$, the problem of a hanging chain is equivalent to find a curve that minimizes the center of mass where the weight of the curve  is measured with respect to a straight-line of $\r^2$. The variational problem is posed among all curves of $\r^2$ with the same endpoints and the same length. Here the curve is viewed as an ideal incompressible chain of constant density.  In the $2$-dimensional  de Sitter space $\s_1^2$,  the   {\it catenary problem} asks what is the shape of a hanging chain in $\s_1^2$. It is then necessary to precise what it is the meaning of the `weight' of a curve of $\s_1^2$ indicating what is the reference to measure the center of mass.  If  we are able to solve this problem, the solution curve will be the analog catenary in $\s_1^2$. Now the natural way to proceed is to immerse $\s_1^2$ in the $3$-dimensional de Sitter space $\s_1^3$, and construct the surface of revolution whose generating curve is the above catenary. It is natural to ask if the mean curvature of this surface vanishes constantly on the surface. 

 A first problem that  appears is that we need to distinguish if   curves in $\s_1^2$ and surfaces in $\s_1^3$ are  spacelike or timelike.  A second  problem is how to define the center of mass  of a curve in $\s_1^2$. Here we need to precise the model of the de Sitter space. Let $\l^3$ be the Lorentz-Minkowski space, that is, the vector space $\r^3$ equipped with the Lorentzian metric $\langle,\rangle=(dx)^2+(dy)^2-(dz)^2$, where $(x,y,z)$ stand for the canonical coordinates of $\r^3$.  The $2$-dimensional  de Sitter space is defined as the set  $\s_1^2=\{(x,y,z)\in\r^3:x^2+y^2+z^2=1\}$ endowed with the induced metric from $\l^3$. The de Sitter space $\s_1^2$ is a Lorentzian surface of constant curvature equal to $1$. 
 
 In order to measure the center of mass of a curve of $\s_1^2$, we fix a reference plane $\Pi$ of $\l^3$ and the center of mass is measured using the distance to $\Pi$. Since in $\l^3$ there are three types of planes  depending on the causal character, the plane $\Pi$ can be  spacelike, timelike and degenerate.  Each one of these planes separates $\s_1^2$ in two domains and let us fix one of these domains, which it is denoted by $(\s_1^2)^+$. Let   $\gamma\colon [a,b]\to (\s_1^2)^+$, $\gamma=\gamma(t)$,  be a non-degenerate curve.  Assuming that the density is constantly $1$, the weight of $\gamma$ is defined by
$\int_a^b d(t)\, ds,$
where $d(t)$ is the distance between $\gamma(t)$ and $\Pi$ and $ds$ is the arc-length of $\gamma$. Thus the center of mass of $\gamma$ is 
$$\frac{\int_a^b d(t)\, ds}{\int_a^b \, ds}=\frac{\int_a^b d(t)|\gamma'(t)|\, dt}{\int_a^b |\gamma'(t)|\, dt},$$
where $|\gamma'(t)|=\sqrt{|\langle \gamma'(t),\gamma'(t)\rangle|}$. The variational problem posed in $\s_1^2$ consists to find the curve that minimizes the center of mass in the class of  all curves of $\s_1^2$ with prescribed endpoints and fixed length. Introducing a Lagrange multiplier due to the length constraint, the energy to minimize is 
\begin{equation}\label{ee}
\mathcal{E}_\Pi[\gamma]=\int_a^b (d(t)+\lambda)|\gamma'(t)|\, dt.
\end{equation}
 
 \begin{definition} \label{ds1} 
A critical point of $\mathcal{E}_\Pi$  is called a catenary with respect to the plane $\Pi$.
\end{definition}

 The following  step in our investigation  consists in   constructing a surface of revolution in $\s_1^3$ having a catenary as generating curve. For this, we immersed $\s_1^2$ into the $3$-dimensional de Sitter space $\s_1^3$ as follows. Let   $\l^4$ be the $4$-dimensional Lorentz-Minkowski space where the metric is $(dx_1)^2+(dx_2)^2-(dx_3)^2+(dx_4)^2$ and $(x_1,x_2,x_3,x_4)$ are Cartesian coordinates of $\r^4$. The $3$-dimensional de Sitter space $\s_1^3$ is defined by  
 $$\s_1^3=\{p=(x_1,x_2,x_3,x_4)\in\l^4:\langle p,p\rangle=1\}$$
 endowed with the induced metric from $\l^4$.  The space $\l^3$ is immersed in $\l^4$ by $(x,y,z)\mapsto (x,y,z,0)$. In particular, the $2$-dimensional de Sitter space $\s_1^2$ is immersed in $\s_1^3$ as $\s_1^2\times\{0\}$.   The steps to follow are now clear:
  \begin{enumerate}
  \item Let $\gamma$ be a catenary of $\s_1^2$ with respect to a plane $\Pi$.
  \item Let $L$ be the geodesic obtained by intersecting $\s_1^2$ with $\Pi$, $L=\s_1^2\cap\Pi$.
  \item Let $L\equiv L\times\{0\}\subset\s_1^3$.  Consider $\mathcal{G}$ the one-parameter group of rotations of $\s_1^3$ that leave pointwise fixed $L$. 
  \item Rotate $\gamma\equiv\gamma\times\{0\}$ about $\mathcal{G}$, obtaining a surface of revolution $S_\gamma$ of $\s_1^3$. 
  \item Compute the mean curvature $H$ of $S_\gamma$ and investigate if $H$ is $0$.
  \end{enumerate}
  
 As a consequence of all this work, the main result can be stated as follows.

\begin{theorem}\label{t-main} Let $M$ be a non-degenerate surface of revolution in $\s_1^3$ whose rotation axis is $L=\s_1^2\cap \Pi$.  Then $M$ is a minimal surface if and only if its generating curve is a catenary with respect to $\Pi$. 
\end{theorem}

This theorem provides a variational characterization of the generating curves of rotational minimal surfaces and it   shows that our approach to catenaries  in the de Sitter space  extends the classical notions of Euclidean catenary and catenoid. 
 
 The proof of Theorem \ref{t-main} is separated in three cases depending in   the three types of surfaces of revolution of $\s_1^3$. Within a more general context, these surfaces are  described in \cite{cd}. To be precise, fix a $2$-plane $P^2\subset \l^4$ and consider the one-parameter group $\mathcal{G}$ of isometries of $\l^4$ that leaves $P^2$  pointwise fixed. Let $P^3$ be a $3$-dimensional subspace of $\l^4$ such that $P^2\subset P^3$ and consider $\gamma$  a non-degenerate curve in $\s_1^3\cap (P^3-P^2)$. The surface  $S_\gamma$  obtained by acting $\mathcal{G}$ on $\gamma$   is called a surface of revolution of $\s_1^3$ and $\gamma$ is called its generating curve. This surface is said to be spherical (resp. hyperbolic, parabolic) if the induced metric on $P$ from $\l^4$ is Riemannian (resp. Lorentzian, degenerate).  Explicit expressions of the groups $\mathcal{G}$ and of the parametrizations of the rotational surfaces will be described in the following sections: spherical case (Section \ref{sec2}), hyperbolic case (Section \ref{sec3}) and parabolic case (Section \ref{sec4}).  

Following the above scheme, in each one of the three sections, it is defined the energy $\mathcal{E}_\Pi$ which depends on the distance to the reference plane $\Pi\subset\s_1^2$. Catenaries will be obtained as solutions of the Euler-Lagrange equation of these energies (Theorems \ref{t11}, \ref{t21} and \ref{t41}). Next, catenaries will be also characterized   as solutions of a prescribing curvature equation involving the distance to the plane $\Pi$ and the angle that make the principal normal vector of the curve with the unit vector field orthogonal to $\Pi$ (Theorems \ref{t12}, \ref{t22}  and \ref{t42}). Finally and answering to the question proposed in this paper, we prove Theorem \ref{t-main} in each one of these cases (Theorems \ref{t13}, \ref{t23}  and \ref{t43}).  A last Section \ref{sec5} will investigate the concept of intrinsic catenary where the center of mass is calculated using the (intrinsic) distance of $\s_1^2$.

\section{The catenary problem: spherical case}\label{sec2}

The first case to investigate are the catenaries of $\s_1^2$ where the center of mass is measured with respect to a spacelike plane. Without loss of generality,  we can assume that this plane is the plane $\Pi_{xy}$  of equation $z=0$. The catenary problem  consists into find the shape of a curve contained in $(\s_1^2)_{xy}^+=\{(x,y,z)\in \s_1^2:z>0\}$  whose shape is obtained as a critical point of the center of mass with respect to $\Pi_{xy}$.  The distance of a point $(x,y,z)\in(\s_1^2)_{xy}^+$ to $\Pi$ is $z$.

In order to obtain a manageable expression of this energy, let  us parametrize the de Sitter space  $\s_1^2$   by $\Psi\colon\r^2\to\s_1^2$, where
$$\Psi(u,v)=(\cosh(u)\cos{v},\cosh(u)\sin{v},\sinh(u)).$$
Let $\gamma\colon [a,b]\to(\s_1^2)_{xy}^+$ be a non-degenerate curve  and let  $\gamma(t)=\Psi(u(t),v(t))$. Then $\langle\gamma'(t),\gamma'(t)\rangle=v'^2\cosh(u)^2-u'^2$. If $\gamma$ is spacelike (resp. timelike) this number is positive (resp. negative). Along this paper, we will write   $\epsilon(v'^2\cosh(u)^2-u'^2)>0$ where $\epsilon=1$ (resp.  $\epsilon=-1$) if $\gamma$ is spacelike (resp. timelike). Then $|\gamma'(t)|=\sqrt{\epsilon(v'^2\cosh(u)^2-u'^2)}$ and $|\gamma'(t)|^2=\epsilon\langle\gamma'(t),\gamma'(t)\rangle$.

The distance of $\gamma(t)$ to $\Pi_{xy}$ is $\sinh u(t)$. According to \eqref{ee},  the energy of the  catenary problem with respect to the plane $\Pi_{xy}$ is 
\begin{equation*}
\mathcal{E}_{xy}[\gamma] =\int_a^b (\sinh(u)+\lambda) \sqrt{\epsilon(v'^2\cosh(u)^2-u'^2)}\, dt,
\end{equation*}
where $\lambda$ is a Lagrange multiplier due to the length of the curve being constant.

We will see that the critical points of $\mathcal{E}_{xy}$ will be expressed in terms of the curvature $\kappa$ of $\gamma$. Here $\kappa$ is the curvature of $\gamma$  viewed as a curve in the space $\s_1^2$, or in other words, $\kappa$ is the geodesic curvature of $\gamma$. For the rest of  computations of this paper, it is necessary to have an expression of $\kappa$ when $\gamma$ is written as $\gamma(t)=\Psi(u(t),v(t))$.

\begin{proposition} Let $\gamma(t)=\Psi(u(t),v(t))$ be a non-degenerate curve in $\s_1^2$. Then its (geodesic) curvature $\kappa$ is 
\begin{equation}\label{eqk}
\kappa=\epsilon
\frac{v'(v'^2\sinh(u)\cosh(u)^2-2u'^2\sinh(u))-\cosh(u)(u'v''-v'u'')}{|\gamma'(t)|^3}.
\end{equation}
\end{proposition}

\begin{proof} Let $\xi(p)=-p$, $p\in \s_1^2$, be the unit normal vector of $\s_1^2$. Then 
$$\kappa=\epsilon\frac{\langle\gamma'',\gamma'\times \xi(\gamma)\rangle}{|\gamma'(t)|^3}=\epsilon\frac{\mbox{det}(\gamma,\gamma',\gamma'')}{|\gamma'(t)|^3},$$
and \eqref{eqk} follows by a straightforward calculation.
\end{proof}

For further purposes, we need the following computations. First, 
\begin{equation}\label{k1}
\begin{split}\frac{d}{dt}\left(\frac{u'}{|\gamma'|}\right)&=\frac{\epsilon v'\cosh(u)}{|\gamma'|^3}(-u'^2v'\sinh(u)+\cosh(u)(v'u''-u'v''))\\
&=\frac{\epsilon v'\cosh(u)}{|\gamma'|^3}(u'^2v'\sinh(u)-v'^3\sinh(u)\cosh(u)^2+\epsilon \kappa|\gamma'|^3)\\
&=v'\cosh(u)\left(\kappa-\frac{v'\sinh(u)}{|\gamma'|}\right),
\end{split}
\end{equation}
where in the second identity we have replaced the expression $v'u''-u'v''$ in terms of $\kappa$ thanks to \eqref{eqk}. A similar computation yields
\begin{equation}\label{k2}
\begin{split}\frac{d}{dt}\left(\frac{v'\cosh(u)}{|\gamma'|}\right)&=u'\left(\kappa-\frac{v'\sinh(u)}{|\gamma'|}\right)\end{split}
\end{equation}

\begin{theorem} \label{t11}
Let $\gamma(t)=\Psi(u(t),v(t))$ be a non-degenerate curve in $(\s_1^2)_{xy}^+$. Then $\gamma$ is a  catenary with respect to $\Pi_{xy}$ if and only if its curvature $\kappa$ satisfies 
\begin{equation}\label{c11}
\kappa=-\frac{v'\cosh(u)^2}{(\sinh(u)+\lambda)|\gamma'|}.
\end{equation}
\end{theorem}

\begin{proof} The Euler-Lagrange equations of $\mathcal{E}_{xy}$ are calculated using 
 \begin{equation}\label{sel}
\frac{\partial J}{\partial u}-\frac{d}{dt} \left(\frac{\partial J}{\partial u'}\right)=0,\quad 
 \frac{\partial J}{\partial v}-\frac{d}{dt} \left(\frac{\partial J}{\partial v'}\right)=0,
\end{equation}
where $J=J[u,v]$ is the integrand of $\mathcal{E}_{xy}$, 
$$J [u,v]=(\sinh(u)+\lambda) \sqrt{\epsilon(v'^2\cosh(u)^2-u'^2)}.$$
Equations \eqref{sel} are, respectively, 
$$\frac{ v'^2\cosh(u)\left(\cosh(u)^2+\sinh(u)(\sinh(u)+\lambda)\right)}{|\gamma'|}+(\sinh(u)+\lambda)\frac{d}{dt}\left(\frac{u'}{|\gamma'|}\right)=0.$$
$$\left(\cosh(u)^2+\sinh(u)(\sinh(u)+\lambda)\right)\cosh(u)\frac{u'v'}{|\gamma'|}+\sinh(u)\cosh(u) \frac{d}{dt}\left(\frac{v'\cosh(u)}{|\gamma'|}\right)=0.$$
 Using \eqref{k1} and \eqref{k2},  we obtain  
$$v'\cosh(u)\left(\frac{v'\cosh(u)^2}{|\gamma'|}+\kappa(\sinh(u)+\lambda)\right)=0.$$
$$u'\cosh(u)\left(\frac{v'\cosh(u)^2}{|\gamma'|}+\kappa(\sinh(u)+\lambda)\right)=0.$$
This yields the result because $u'$ and $v'$ cannot simultaneously vanish.
\end{proof}

\begin{remark} The case $v'=0$ is equivalent to $\kappa=0$ and $\gamma$ is a meridian $\gamma(t)=\Psi(t,v_0)$ for some constant $v_0\in\r$. This situation also occurs in the Euclidean case: see Remark \ref{rem2} below.
\end{remark}

Notice that the Lagrangian $J[u,v]$ does not depend on $v$. Thus the second equation of \eqref{sel} gives a first integration of $\gamma$, namely, 
$$\frac{\partial J}{\partial v}=\frac{\epsilon v'\cosh(u)^2(\sinh(u)+\lambda)}{|\gamma'|}=c,$$
for some non-zero constant $c\in\r$. Without loss of generality, we can assume that $\gamma(t)=\Psi(u(t),t)$. From the above equation, we deduce  
\begin{equation}\label{first}
u(t)=m\pm\frac{1}{c}\int^t\cosh(u)\sqrt{c^2-\epsilon\cosh(u)^2(\sinh(u)+\lambda)}\, dt,\quad c,m\in\r.
\end{equation}

We characterize the   catenaries with respect to $\Pi_{xy}$  in terms of the angle that makes the principal normal vector ${\bf n}$ of $\gamma$ with a vector field of $\l^3$.  The vector $\mathbf{n}$ is understood as a unitary tangent  of  $\s_1^2$ at $\gamma$ which is orthogonal to $\gamma'$.  On the other hand, the vector field of $\l^3$ is the unit vector field orthogonal to the plane $\Pi_{xy}$. Define  $Z=\partial_z\in\mathfrak{X}(\l^3)$.

\begin{theorem}\label{t12} Let $\gamma(t)=\Psi(u(t),v(t))$ be a non-degenerate curve in $(\s_1^2)_{xy}^+$. Then $\gamma$ is a catenary with respect to $\Pi_{xy}$ if and only if its curvature $\kappa$ satisfies 
\begin{equation}\label{c12}
\kappa(t)=\frac{\langle{\bf n}(t), Z\rangle}{d_{xy}(t)+\lambda},
\end{equation}
where $d_{xy}(t)$ is the distance of $\gamma(t)$ to the plane $\Pi_{xy}$.
\end{theorem}
\begin{proof} 
Since the unit normal vector $\xi$ to $(\s_1^2)_{xy}^+$ along $\gamma$ is $\xi(\gamma)=-\gamma$,   the principal normal vector ${\n}$ is 
$${\bf n}=\frac{ \xi(\gamma)\times \gamma'}{|\gamma'|}=-\frac{\gamma\times\gamma'}{|\gamma'|}.$$
Thus
$$\langle{\bf n}(t),Z\rangle=-\frac{\mbox{det}(\gamma,\gamma',\partial_z)}{ |\gamma'|}= -\frac{v'\cosh(u)^2}{|\gamma'|}.$$
Since $d_{xy}=\sinh(u)$, it follows \eqref{c12} from \eqref{c11}.
\end{proof}

\begin{remark}\label{rem2}
Equation \eqref{c12} is analogous to the Euclidean case. In the catenary problem in $\r^2$,    the reference line is assumed to be the $x$-axis. The energy to minimize is   $[y]\mapsto \int_a^b (y+\lambda)\sqrt{1+y'^2}\, dx$ for curves $y=y(x)$, $x\in [a,b]$, $y(x)>0$. A critical point $y$ is characterized by the equation
\begin{equation}\label{euclideo}
\frac{y''}{1+y'^2}=\frac{1}{y+\lambda},
\end{equation}
 whose solution   is the catenary
\begin{equation}\label{cat}
y(x)=\frac{1}{c}\cosh(cx+a)-\lambda,\quad a,c\in\r, c>0.
\end{equation}
 The curvature of $y(x)$ is $y''/(1+y'^2)^{3/2}$ and the principal normal vector $\mathbf{n}$ is $(-y',1)/\sqrt{1+y'^2}$. If $W(x,y)=\partial_y$ is the unit vector field of $\r^2$ in the direction of the gravity,  then equation \eqref{euclideo} is equivalent to 
$$\kappa(x)=\frac{\langle \mathbf{n}(x),W\rangle}{d(x)+\lambda},$$
where $d(x)=y$ is the distance of $(x,y(x))$ to the $x$-axis. Notice that vertical straight-lines of $\r^2$ are solutions of this equation because $\kappa=0$ and $\langle{\bf n},W\rangle=0$.
\end{remark}

Once the catenaries of $\s_1^2$ with respect to the plane $\Pi_{xy}$ have been established, we  answer to the Question posed in the Introduction. We will rotate a catenary curve with respect to the plane $\Pi_{xy}$ about the geodesic $L_{xy}=\Pi_{xy}\cap\s_1^2$.  For this, let us immerse  $\s_1^2$ into $\s_1^3$ considering a Riemannian factor in the fourth coordinate such as it was explained in Section \ref{s1}.    Consider the geodesic $L_{xy}\subset\s_1^2\times\{0\}\subset\s_1^3$, which will be the rotation axis. The one-parametric group of rotations of $\s_1^3$ that leave pointwise the geodesic $L_{xy}$ is $\mathcal{G}_{xy}=\{\mathcal{R}_s^{xy}:s\in\r\}$, where 
 $$\mathcal{R}_s^{xy}= \left(\begin{array}{cccc} 1&0&0&0\\ 0&1&0&0\\ 0&0&\cosh(s)&\sinh(s)\\ 0&0& \sinh(s) &\cosh(s)\end{array}\right).$$
 Let $\gamma=\gamma(t)$ be a curve in $(\s_1^2)_{xy}^+\subset\s_1^3$. Let $S_\gamma^{xy}=\{\mathcal{R}^{xy}_s\cdot\gamma(t)\colon s\in\r, t\in [a,b]\}$ be the surface of revolution obtained by the orbit of $\gamma$ under the group $\mathcal{G}_{xy}$. Since the plane $\Pi_{xy}$ containing the rotation axis is spacelike, the surface of revolution is of spherical type according to the terminology of \cite{cd}.

 \begin{theorem}\label{t13}
  The surface of revolution $S_\gamma^{xy}$ is minimal if and only if $\gamma$ is a catenary  with respect to $\Pi_{xy}$  for the Lagrange multiplier $\lambda=0$.
 \end{theorem}
 
 \begin{proof} 
 In order to simplify the arguments, we can assume without loss of generality that  $\gamma(t)=\Psi(u(t),t)$. Let us see $\gamma(t)\equiv  (\Psi(u(t),t),0)$ as a curve in $\s_1^3$.   A parametrization of $S_\gamma^{xy}$ is 
 $$\mathbf{r}(t,s)=(\cosh(u)\cos(t),\cosh(u)\sin(t),\cosh(s)\sinh(u),\sinh(s)\sinh(u)).$$
 It is necessary the expression of  the mean curvature $H$ of $S_\gamma^{xy}$ in terms of the parametrization $\mathbf{r}$. As usually, let $\{E,F,G\}$ and $\{h_{11},h_{12},h_{22}\}$ be the coefficients of the first and second fundamental form of $S_\gamma^{xy}$ for the parametrization $\mathbf{r}$: 
 $$E=\langle \mathbf{r}_{t},\mathbf{r}_{t}\rangle,\quad F=\langle \mathbf{r}_{t},\mathbf{r}_{s}\rangle,\quad G=\langle \mathbf{r}_{s},\mathbf{r}_{s}\rangle,$$
 $$h_{11}=\langle N,\mathbf{r}_{tt}\rangle,\quad h_{12}=\langle N,\mathbf{r}_{ts}\rangle,\quad h_{22}=\langle N,\mathbf{r}_{ss}\rangle.$$
 Then 
 $$H=\frac{\delta}{2}\frac{Eh_{22}-2Fh_{12}+G h_{11}}{EG-F^2},$$
 where $\delta=1$ if $S_\gamma^{xy}$ is spacelike and $\delta=-1$ if $S_\gamma^{xy}$ is timelike. Notice that if $\epsilon=-1$, then necessarily $\delta=-1$, but if $\epsilon=1$, then $\delta$ may be $1$ or $-1$. In the proof of this theorem, and in the subsequent sections, the coefficient $F$ vanishes, so 
 \begin{equation}\label{hh}
 H=\frac{\delta}{2}\frac{Eh_{22} +G h_{11}}{EG}.
 \end{equation}
 In fact, we will investigate when $S_{\gamma}^{xy}$ is minimal, so our interest focuses under what hypotheses on $\gamma$  we have $Eh_{22}+Gh_{11}$ is $0$ on the surface. 
 
 The calculation of $N$ is obtained knowing that $N$ is not only orthogonal to $\mathbf{r}_{t}$ and $\mathbf{r}_{s}$, but also to $\mathbf{r}$ since $N$ is a tangent vector of $\s_1^3$.   A straightforward computations leads to, 
  $$N=\frac{1}{|\gamma'|}\left(\begin{array}{l}\cos (t) \sinh(u)\cosh(u)-u'\sin (t)   \\  u' \cos (t)  +\sin(t)\sinh(u)\cosh(u) \\
  \cosh (s) \cosh(u)^2\\ \sinh (s) \cosh(u)^2\end{array}   \right).$$
  In particular, $\langle N,N\rangle=-\epsilon$. We also have   
 $$E=  \cosh (u)^2-u'^2=\epsilon|\gamma'|^2,\quad G=\sinh(u)^2.$$
 On the other hand, the coefficients $h_{11}$ and $h_{22}$ of the second fundamental form are  
 \begin{equation*}
 \begin{split}
 h_{11}&=\frac{-\cosh(u)u''+2 u'^2 \sinh(u)-\sinh(u) \cosh(u)^2}{|\gamma'|},\\
 h_{22}&=-\frac{\sinh(u) \cosh(u)^2}{|\gamma'|}.
 \end{split}
 \end{equation*}
 Then 
$$
H=-\frac{\delta\sinh(u)}{4|\gamma'|EG}\left((1-3\cosh(2u))u'^2+\cosh(u)(\cosh(u)+\cosh(3u)+2\sinh(u)u'')\right).
$$
In this expression of $H$, the term $u''$ is replaced in function of the curvature $\kappa$. So, from  \eqref{eqk} we have
$$u''=\frac{\epsilon\kappa|\gamma'|^3-\sinh(u)\cosh(u)^2+2u'^2\sinh(u)}{\cosh(u)}.$$
Then
 \begin{equation}\label{h1}
 H=-\frac{\delta\epsilon\sinh(u)|\gamma'|}{2EG}\left(\cosh(u)^2+\kappa\sinh(u)|\gamma'|\right).
 \end{equation}
 Then $H=0$ if and only if $\cosh(u)^2+\kappa\sinh(u)|\gamma'|$, proving the result thanks to \eqref{c11}. 
 \end{proof}

 \begin{remark} The fact that $\lambda$ must be $0$ in Theorem \ref{t13} in order to ensure that $S_\gamma^{xy}$ is   minimal is expectable because the same situation occurs in Euclidean space $\r^3$. 
 If we immerse $\r^2$ into $\r^3$ by $(x,y)\mapsto (x,0,y)$ and if   the catenary $y(x)$ given in \eqref{cat} is rotated about the $x$-axis of $\r^3$, the mean curvature $H$ of the surface is 
 $$H=\frac12\left(\frac{1}{y\sqrt{1+y'^2}}-\frac{y''}{(1+y'^2)^{3/2}}\right)= \frac{\lambda c^2}{2\cosh(cx+a)^2(\cosh(cx+a)-\lambda c)}.$$
 Thus $H=0$ if and only if $\lambda=0$.
 \end{remark}

 \section{The catenary problem: hyperbolic case}\label{sec3}

In the case of the catenary problem with respect to a Lorentzian plane, we have two possible choices of Lorentzian coordinate planes. After a rigid motion of $\s_1^2$, we can assume that the plane is the plane $\Pi_{xz}$ of equation $y=0$.  In this section it will be studied the curves of $\s_1^2$ that are critical points of the center of mass when it  is measured with respect to the  plane $\Pi_{xz}$. The distance of $(x,y,z)\in\s_1^2$ to $\Pi_{xz}$ is $|y|$. Let $(\s_1^2)_{xz}^+=\{(x,y,z)\in\s_1^2:y>0\}$. Let $\gamma=\gamma(t)$, $t\in [a,b]$, be a curve in $(\s_1^2)_{xz}^+$ and let $\gamma(t)=\Psi(u(t),v(t))$. The expression of the energy \eqref{ee} is  now
$$\mathcal{E}_{xz}[\gamma]=\int_a^b (\cosh(u)\sin(v)+\lambda) \sqrt{\epsilon(v'^2\cosh(u)^2-u'^2)}\, dt.$$
 
\begin{theorem}  \label{t21}
Let $\gamma(t)=\Psi(u(t),v(t))$ be a non-degenerate curve in $(\s_1^2)_{xz}^+$. Then $\gamma$ is a  catenary with respect to $\Pi_{xz}$ if and only if its curvature $\kappa$ satisfies 
\begin{equation}\label{c21}
\kappa=-\frac{u'\cos(v)+v'\sinh(u)\cosh(u) \sin(v)}{(\cosh(u)\sin(v)+\lambda)|\gamma'|}.
\end{equation}
\end{theorem}

\begin{proof} 
We compute the Euler-Lagrange equations \eqref{sel} for the energy $\mathcal{E}_{xz}$. The first equation  
$\frac{\partial J}{\partial u}-\frac{d}{dt} \left(\frac{\partial J}{\partial u'}\right)=0$ writes as
\begin{equation*}
\begin{split}
&\frac{v'\cosh(u)}{|\gamma'|}\left(v'\cosh(u)\sinh(u)\sin(v)+v'\sinh(u)(\cosh(u)\sin(v)+\lambda)+u'\cos(v)\right)\\
&=-(\cosh(u)\sin(v)+\lambda)\frac{d}{dt}\left(\frac{u'}{|\gamma'|}\right).
\end{split}
\end{equation*}
Using \eqref{k1}, this equation is simplified into
  $$\frac{v'\sinh(u)\cosh(u) \sin(v)+u'\cos(v)}{|\gamma'|}=-(\cosh(u)\sin(v)+\lambda) \kappa.$$
The second equation of \eqref{sel} coincides with the above one, obtaining \eqref{c21}.
\end{proof}

As in Theorem \ref{t12}, we characterize the catenaries  with respect to $\Pi_{xz}$ in terms of the   angle that makes the principal normal vector ${\bf n}$ of $\gamma$ with a vector field of $\l^3$. Consider the unit vector field $Y=\partial_y\in\mathfrak{X}(\l^3)$. Notice that $Y$  is   orthogonal to the plane $\Pi_{xz}$.

\begin{theorem}\label{t22} Let $\gamma(t)=\Psi(u(t),v(t))$ be a non-degenerate curve in $(\s_1^2)_{xz}^+$. Then $\gamma$ is a catenary with respect to $\Pi_{xz}$ if and only if its curvature $\kappa$ satisfies 
\begin{equation}\label{c22}
\kappa(t)=-\frac{\langle{\bf n}(t), Y\rangle}{d_{xz}(t)+\lambda},
\end{equation}
where $d_{xz}(t)$ is the distance of $\gamma(t)$ to the plane $\Pi_{xz}$.
\end{theorem}

\begin{proof} Following the same arguments as in Theorem \ref{t12}, we have 
$$\langle{\bf n}(t),Y\rangle=-\frac{ \mbox{det}(\gamma,\gamma',\partial_y)}{|\gamma'|}= \frac{v'\sinh(u)\cosh(u) \sin(v)+ u'\cos(v)}{|\gamma'|}.$$
Since $d_{xz}=\cosh(u)\sin(v)$, it follows the result from \eqref{c21}.
\end{proof}
 
Finally we construct the surfaces of revolution in $\s_1^3$ of hyperbolic type using the catenaries with respect to $\Pi_{xz}$. Consider the geodesic $L_{xz}=\s_1^2\cap\Pi_{xz}$ viewed as a curve in $\s_1^3$. This geodesic $L_{xz}$ will be the rotation axis of a rotational surface of $\s_1^3$ of hyperbolic type because the plane $\Pi_{xz}$ is Lorentzian.  The one-parameter family of rotations fixing $L_{xz}$  is the group $\mathcal{G}_{xz}=\{R_s^{xz}:s\in\r\}$, where
 $$ R_s^{xz}: =\left(\begin{array}{cccc}1&0&0&0\\ 0&\cos s&0&-\sin s\\ 0&0&1&0\\ 0&\sin s&0&\cos s\end{array}\right).$$
 If  $\gamma=\gamma(t)$ is a curve in $(\s_1^2)_{xz}^+$, let $S_\gamma^{xz}$ denote the surface of revolution of $\s_1^3$ obtained by rotating $\gamma$  with respect to $L_{xz}$, $S_\gamma^{xz}=\{\mathcal{R}_s^{xz}\cdot\gamma(t):s\in\r, t\in [a,b]\}$. 
 
 \begin{theorem} \label{t23}
The surface of revolution $S_\gamma^{xz}$ is minimal if and only if $\gamma$ is a catenary with respect to $\Pi_{xz}$ for the Lagrange multiplier $\lambda=0$.
\end{theorem}
\begin{proof} 
Without loss of generality, we assume that $\gamma$ is parametrized by $\gamma(t)=\Psi(u(t),t)$, $t\in [a,b]$. 
 A parametrization of $S_\gamma^{xz}$ is 
 $$\mathbf{r}(t,s) =( \cos(t) \cosh(u),\sin(t) \cosh(u)\cos(s),\sinh(u),\sin(s) \sin(t) \cosh(u)).$$
 The unit normal is 
$$N=\frac{1}{|\gamma'|}\left(\begin{array}{c}
\cos (t) \sinh (u) \cosh (u)-\sin (t) u'\\
\cos (s) \left(\cos (t) u'+\sin (t) \sinh (u) \cosh (u)\right)\\
\cosh ^2(u)\\
\sin (s) \left(\cos (t) u'+\sin (t) \sinh (u) \cosh (u)\right)
\end{array}\right),$$
and $\langle N,N\rangle=-\epsilon$. 
The coefficients of the first fundamental form are $F=0$ and 
$$E=\cosh(u)^2-u'^2=\epsilon|\gamma'|^2,\quad G= \sin(t)^2\cosh(u)^2.$$
We also have
\begin{equation*}
\begin{split}
h_{11}&=\frac{2u'^2\sinh(u)-\cosh(u)(u''+\sinh(u)\cosh(u))}{|\gamma'|},\\
 h_{22}&=-\frac{\cosh(t)\sin(t)(u'\cos(t)+\sinh(u)\cosh(u)\sin(t))}{|\gamma'|}.
 \end{split}
 \end{equation*}
Using \eqref{hh}, we have 
\begin{equation*}
\begin{split}
 H&=\frac{\delta\cosh(u)\sin(t)}{4EG|\gamma'|}\\
 & \Big(-2 \sin (t) \cosh(u)^2 \left(u''+\sinh (2 u)\right)+2 \cos (t) u'^3\\
 &+3 \sin (t) u'^2 \sinh (2 u)-2 \cos (t) u' \cosh ^2(u)\Big).
 \end{split}
 \end{equation*}
Using \eqref{eqk}, we replace $u''$ in terms of the curvature $\kappa$, obtaining
$$
H=-\frac{\delta\epsilon|\gamma'|}{2EG}\left(u'\cos(t)+\sinh(u)\cosh(u)\sin(t)+\cosh(u)\sin(t)|\gamma'|\kappa\right).
$$
The result is now immediate from \eqref{c21}.
\end{proof}


\section{The catenary problem: parabolic case} \label{sec4}

This section is devoted to consider catenaries of $\s_1^2$ when the center of mass is calculated with the distance with respect to a degenerate plane. Without loss of generality, let $\Pi_{y-z}$ be the plane of equation $y-z=0$. The distance of    a point  $(x,y,z)$ to $\Pi_{y-z}$ is   $|y-z|$ up to the factor $\sqrt{2}$. Let $(\s_1^2)_{y-z}^+$ be half-space $\{(x,y,z)\in\s_1^2:y-z>0\}$. For a curve $\gamma(t)=\Psi(u(t),v(t))$, $t\in  [a,b]$, contained in $(\s_1^2)_{y-z}^+$, define   the energy   
$$\mathcal{E}_{y-z}[\gamma]=\int_a^b (\cosh(u)\sin(v)-\sinh(u) +\lambda)\sqrt{\epsilon(v'^2\cosh(u)^2-u'^2)}\, dt.$$
where $\lambda$ is a Lagrange multiplier.  

\begin{theorem}  \label{t41}
Let $\gamma(t)=\Psi(u(t),v(t))$ be a non-degenerate curve in $(\s_1^2)_{y-z}^+$. Then $\gamma$ is a  catenary with respect to $\Pi_{y-z}$ if and only if its curvature $\kappa$ satisfies 
\begin{equation}\label{c41}
\kappa= -\frac{v'\cosh(u)(\sinh(u)\sin(v)-\cosh(u))+u' \cos(v)}{(\cosh(u)\sin(v)-\sinh(u)+\lambda)|\gamma'|}.
\end{equation}
\end{theorem}

\begin{proof} 
The computations are similar as in Theorem \ref{t11}. For this, it suffices to compute $\frac{\partial J}{\partial u}-\frac{d}{dt} \left(\frac{\partial J}{\partial u'}\right)=0$. Then we have 
\begin{equation*}
\begin{split}
&\frac{v'^2\left(\cosh(u)^2(\sinh(u)\sin(v)-\cosh(u))+\sinh(u)\cosh(u)(\cosh(u)\sin(v)-\sinh(u)+\lambda)\right)}{|\gamma'|}\\
&=
-\frac{\cosh(u)\cos(v)u'v'}{|\gamma'|}-(\cosh(u)\sin(v)-\sinh(u)+\lambda)\frac{d}{dt}\left(\frac{u'}{|\gamma'|}\right).\end{split}
\end{equation*}
Using \eqref{k1}, and simplifying, 
\begin{equation*}
\begin{split}
&\frac{v'^2\cosh(u)^2(\sinh(u)\sin(v)-\cosh(u))+u'v'\cosh(u)\cos(v)}{|\gamma'|}\\
&=-(\cosh(u)\sin(v)-\sinh(u)+\lambda)v'\cosh(u)\kappa.
\end{split}
\end{equation*}

 \end{proof}
 
We characterize the catenaries of Theorem \ref{t41}  in terms of the   angle that makes the  principal normal vector ${\bf n}$ of $\gamma$ with the vector field $T=\partial_y+\partial_z\in\mathfrak{X}(\l^3)$. This vector field is      orthogonal to the plane $\Pi_{y-z}$.

\begin{theorem}\label{t42} Let $\gamma(t)=\Psi(u(t),v(t))$ be a non-degenerate curve in $(\s_1^2)_{y-z}^+$. Then $\gamma$ is a catenary with respect to $\Pi_{y-z}$ if and only if its curvature $\kappa$ satisfies 
\begin{equation}\label{c42}
\kappa(t)=-\frac{\langle{\bf n}(t), T\rangle}{d_{y-z}(t)+\lambda},
\end{equation}
where $d_{y-z}(t)$ is the distance of $\gamma(t)$ to the plane $\Pi_{y-z}$.
\end{theorem}

\begin{proof} It suffices to compute $\langle{\bf n}(t), T\rangle$,  
$$\langle{\bf n}(t),T\rangle=\frac{v' \cosh(u)\left(\sinh(u)\sin(v)-\cosh(u)\right)-u' \cos(v)}{|\gamma'|}$$
and use   \eqref{c41}.
\end{proof}

We now consider rotational surfaces of $\s_1^3$ about the rotation axis $L_{y-z}=\Pi_{y-z}\cap\s_1^2\subset\s_1^3$. Since the plane $\Pi_{y-z}$ is degenerate, the surface of revolution is of parabolic type.     The one-parameter group $\mathcal{G}_{y-z}$ of rotations of $\s_1^3$ fixing $L_{y-z}$  is the group $\mathcal{G}_{y-z}=\{R_s^{y-z}:s\in\r\}$, where 
$$R_s^{y-z}=\left(
\begin{array}{cccc}
 1 & 0 & 0 & 0 \\
 0 & 1-\frac{s^2}{2} & \frac{s^2}{2} & s \\
 0 & -\frac{s^2}{2} & \frac{s^2}{2}+1 & s \\
 0 & -s & s & 1 \\
\end{array}
\right).$$
Let $\gamma=\gamma(t)$ be a curve in $\s_1^2$ and let $S_\gamma^{y-z}=\{\mathcal{R}_s^{y-z}\cdot\gamma(t):s\in\r,t\in [a,b]\}$ the surface of revolution defined by $\gamma$. 

\begin{theorem}\label{t43}
The surface of revolution $S_\gamma^{y-z}$ is  minimal if and only if $\gamma$ is a catenary with respect to $\Pi_{y-z}$ for the Lagrange multiplier $\lambda=0$.
\end{theorem}

\begin{proof}
The computations are simplified again if the generating curve is assumed to be   $\gamma(t)=(\Psi(u(t),t),0)$. The parametrization of $S_\gamma^{y-z}$  is 
 $$\mathbf{r}(t,s) =\left(\begin{array}{l}
 \cos (t) \cosh (u)\\
  \frac{s^2}{2} \sinh (u)-\frac{s^2-2}{2} \sin (t) \cosh (u) \\
\frac{s^2+2}{2} \sinh (u)-\frac{s^2}{2} \sin (t) \cosh (u)\\
 s (\sinh (u)-\sin (t) \cosh (u))
 \end{array}\right).$$
We compute the mean curvature $H$ of $S_{\gamma}^{y-z}$. The unit normal $N$ of $S_{\gamma}^{y-z}$ is 
$$N=\frac{1}{|\gamma'|}\left(\begin{array}{c}
  \cos (t) \sinh (2 u)-2 \sin (t) u'\\
 - \left(s^2-2\right) \cos (t) u'+ s^2 \cosh ^2(u)-\frac{s^2-2}{2} \sin (t) \sinh (2 u)\\
 -s^2 \cos (t) u'+  \left(s^2+2\right) \cosh ^2(u)-\frac{s^2}{2}\sin (t)) \sinh (2 u)\\
 - s \left(2 \cos (t) u'-2 \cosh ^2(u)+\sin (t) \sinh (2 u)\right)
\end{array}\right)$$
and $\langle N,N\rangle=-\epsilon$.
The coefficients of the first fundamental form are $F=0$ and 
$$E=\cosh(u)^2-u'^2=\epsilon|\gamma'|^2,\quad G=(\sinh(u)-\cosh(u)\sin(t))^2.$$
The coefficients of the second fundamental form are $h_{12}=0$ and 
\begin{equation*}
\begin{split}
h_{11}&=\frac{2\sinh(u)u'^2-\sin(u)\cosh(u)^2-\cosh(u)u''}{|\gamma'|},\\
h_{22}&=\frac{(\cosh(u)\sin(t)-\sinh(u))(\cosh(u)^2-\sinh(u)\cosh(u)\sin(t)-u'\cos(t))}{|\gamma'|}.
\end{split}
\end{equation*}
As in the above cases, the mean curvature $H$ is calculated using the formula \eqref{hh}. After substituting $u''$ in function of $\kappa$ thanks to \eqref{eqk}, we obtain 
\begin{equation*}
\begin{split}
H&=\frac{\delta\epsilon(\cosh(u)\sin(t)-\sinh(u))}{2EG}\\
&\Big( \kappa(\cosh(u)\sin(t)-\sinh(u))|\gamma'|+\cosh(u)(\sinh(t)\sin(t)\\
&-\cosh(u))+u'\cos(t)\Big).
\end{split} 
\end{equation*}
Then \eqref{c41} completes the proof.
\end{proof}
InFigure \ref{fig1}  shows pictures of catenaries of $\s_1^2$ with respect to the three types of causal character of the plane reference.

\begin{figure}[hbtp]
\begin{center}
\includegraphics[width=.3\textwidth]{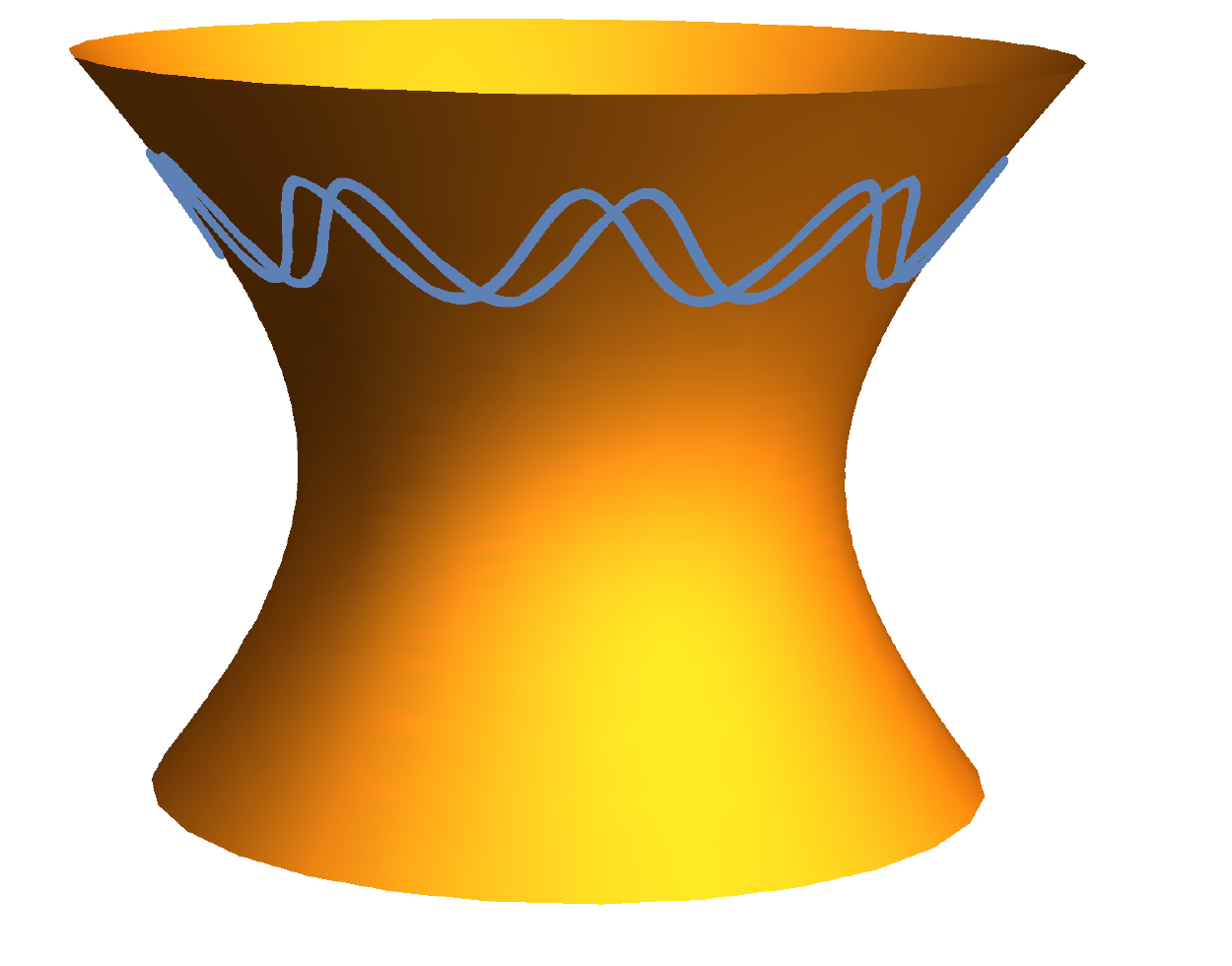} \includegraphics[width=.26\textwidth]{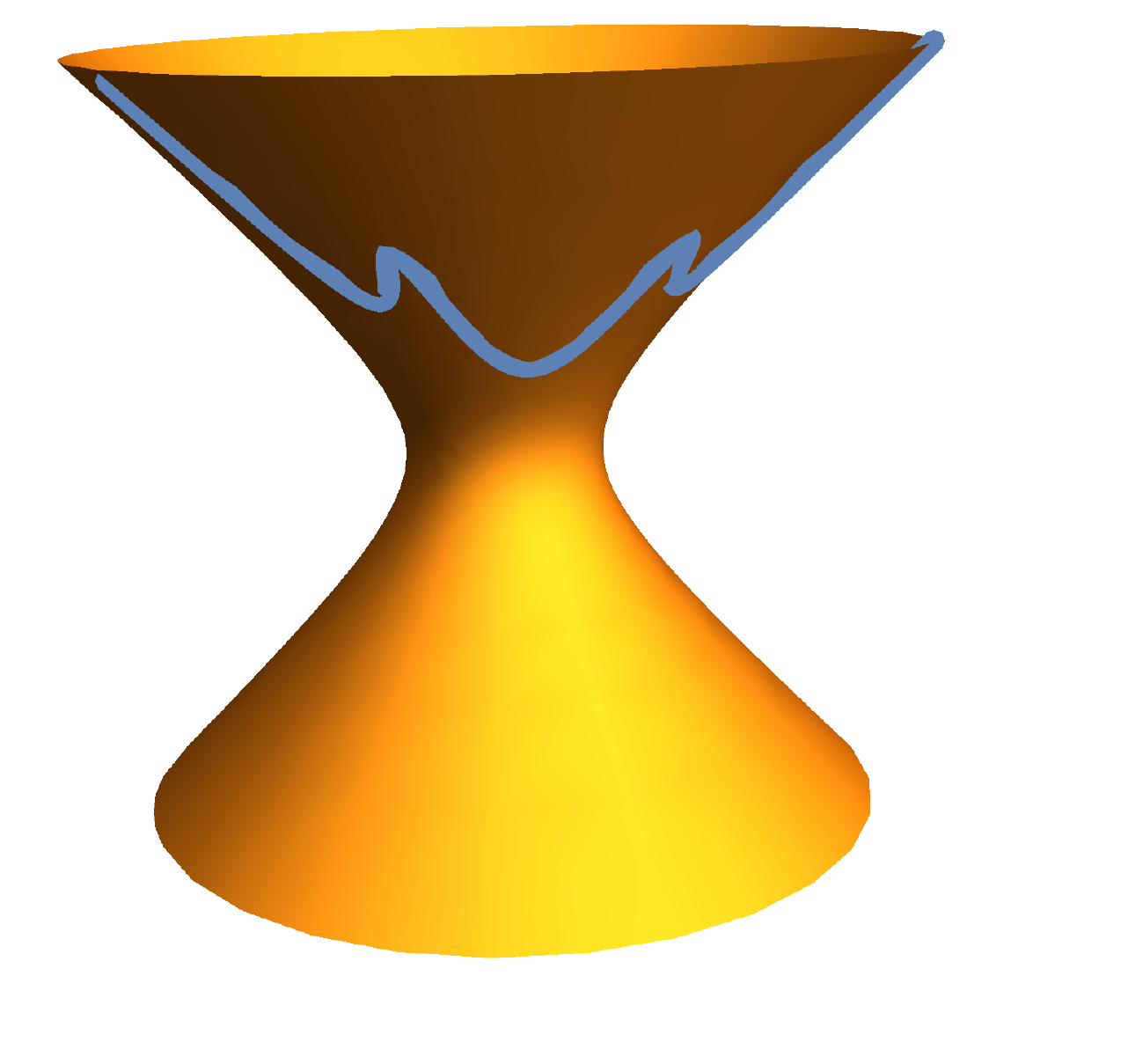} \includegraphics[width=.32\textwidth]{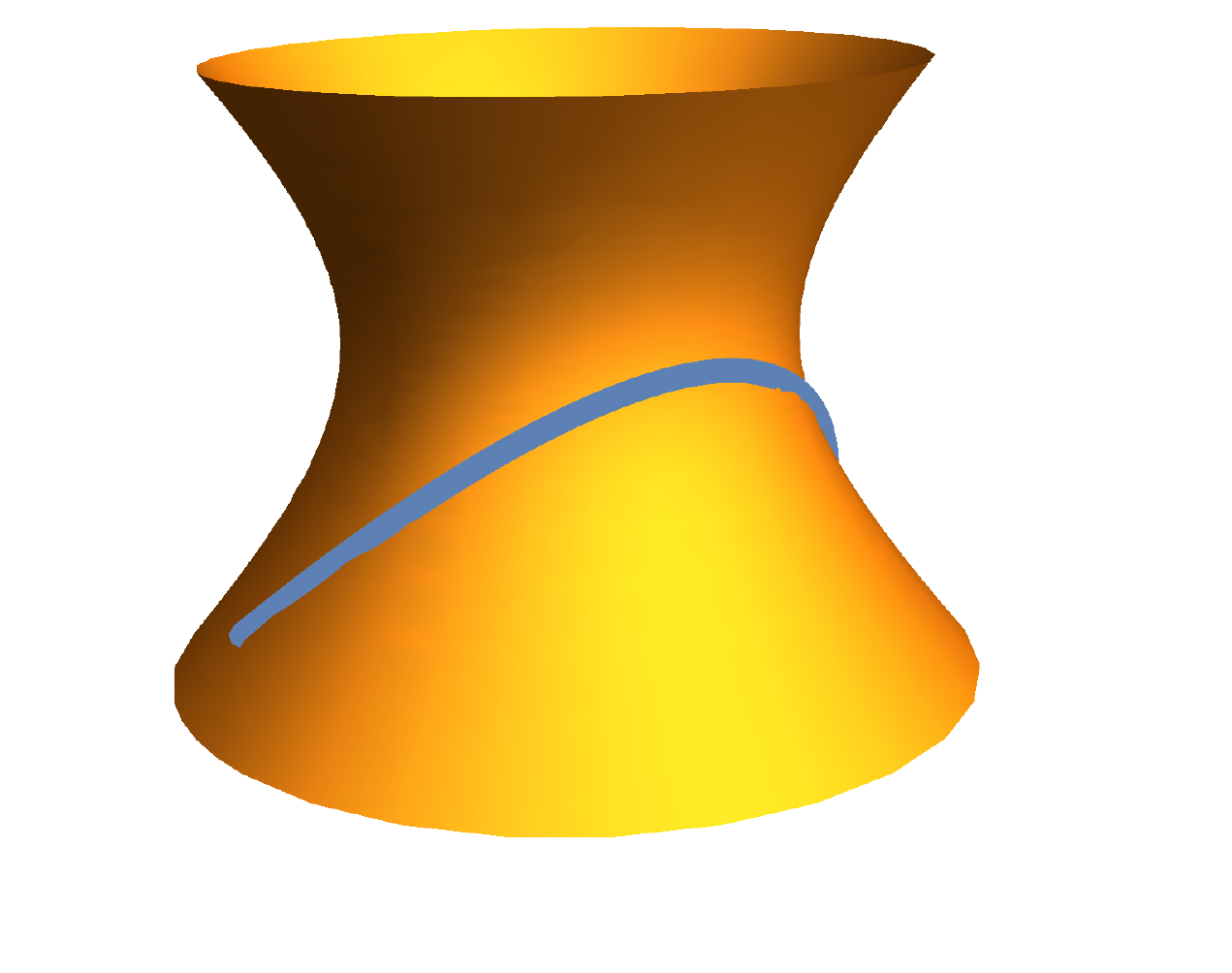} 
\end{center}
\caption{Catenaries of $\s_1^2$: with respect to spacelike plane (left), timelike plane (middle) and degenerate plane (right). }\label{fig1}
\end{figure}

\section{The intrinsic catenary problem}\label{sec5}

The catenaries studied in the previous sections are critical points of energies  that measure the center of mass with respect to  planes of $\l^3$.  However, if one wants to pose an `intrinsic catenary problem' involving only data in $\s_1^2$,  it would be desirable that the distance were measured in $\s_1^2$   with respect to a geodesic of $\s_1^2$. In this section we will investigate this problem in the particular case that this geodesic is spacelike. Without loss of generality, consider the spacelike geodesic $L=\s_1^2\cap\Pi_{xy}$, which will serve as reference line to measure the center of mass of a curve.
 
 The distance of a point $(x,y,z)\in\s_1^2$ from $L$ is computed as follows. Let $(x,y,z)=\Psi(u,v)$. Then the geodesic on $\s_1^2$ passing through to $(x,y,z)$ and orthogonal to $L$ is $\alpha(t)=\Psi(u+t,v)$, $t\in [-u,0]$. The intrinsic distance of $(x,y,z)$ to $L$ is 
$$\int_{-u}^{0}|\alpha'(t)|\, dt=\int_0^{-u}dt=u=\mbox{arcsinh}(z).$$
In particular, it is necessary that  $z\not=0$. Consider the half-space   $(\s_1^2)_{xy}^+=\{(x,y,z)\in \s_1^2:z>0\}$. If $\gamma$ is a curve in $(\s_1^2)_{xy}^+$ parametrized by $\gamma(t)=\Psi(u(t),v(t))$,  the functional energy is defined by 
$$\mathcal{E}_{in}[\gamma]=\int_a^b (u+\lambda)\sqrt{\epsilon(v'^2\cosh(u)^2-u'^2)}\, dt,$$
where $\lambda$ is a Lagrange multiplier that indicates that in the variational problem it is assumed that all curves have fixed length. 
A critical point of $\mathcal{E}_{in}$ is called an {\it intrinsic catenary} of $\s_1^2$ with respect to the geodesic $L$.

\begin{theorem}\label{t1} Let $\gamma(t)=\Psi(u(t),v(t))$ be a non-degenerate curve in $(\s_1^2)_{xy}^+$. Then $\gamma$ is an intrinsic catenary with respect to $L$ if and only if its curvature $\kappa$ satisfies 
\begin{equation}\label{c51}
\kappa=-\frac{v'\cosh(u)}{(u+\lambda)|\gamma'|}.
\end{equation}
\end{theorem}

\begin{proof} We compute the Euler-Lagrange equations of $\mathcal{E}_{in}$ using \eqref{sel}. Both equations are
$$\frac{v'^2\cosh(u)(\cosh(u)+(u+\lambda)\sinh(u))}{|\gamma'|}=-(u+\lambda)\frac{d}{dt}\left(\frac{u'}{|\gamma'|}\right),$$
$$ \frac{u'v'\cosh(u)(\cosh(u)+(u+\lambda)\sinh(u))}{|\gamma'|}+(u+\lambda)\cosh(u)\frac{d}{dt}\left(\frac{v'\cosh(u)}{|\gamma'|}\right)=0.$$

Using the expressions \eqref{k1} and \eqref{k2}, we have
$$v'\cosh(u)\left((u+\lambda)\kappa+\frac{v'\cosh(u)}{|\gamma'|}\right)=0,$$
$$u'\cosh(u)\left((u+\lambda)\kappa+\frac{v'\cosh(u)}{|\gamma'|}\right)=0,$$
proving the result.
\end{proof}

It is possible to  characterize intrinsic catenaries of $\s_1^2$ as solutions a prescribing curvature equation involving the principal normal vector of the curve and vector fields. In contrast to Theorems \ref{t12}, \ref{t22} and \ref{t42}, it is natural that now the vector field is a vector field on $\s_1^2$. Furthermore, this vector field will indicate the direction of how is measuring the distance to $L$. Since this distance is done using the geodesics orthogonal to $L$, that is, the meridians of $\s_1^2$, the vector field will be tangent to the meridians. Hence, define   the vector field $V\in\mathfrak{X}(\s_1^2)$  
\begin{equation*}
V(\Psi(u,v))=\Psi_u=\sinh(u)\cos(v)\partial_x+\sinh(u)\sin(v)\partial_y+\cosh(u)\partial_z.
\end{equation*}

\begin{theorem} \label{t2}Let $\gamma(t)=\Psi(u(t),v(t))$ be a non-degenerate curve in $(\s_1^2)_{xy}^+$. Then $\gamma$ is an intrinsic catenary with respect to $L$ if and only if its curvature $\kappa$ satisfies 
\begin{equation*} 
\kappa(t)=\frac{\langle{\bf n}(t), V\rangle}{d(t)+\lambda},
\end{equation*}
where $d(t)$ is the intrinsic distance of $\gamma(t)$ to $L$.
\end{theorem}

\begin{proof}
Using that  ${\bf n}= - \gamma\times\gamma'/|\gamma'|$, from $\gamma(t)=\Psi(u(t),v(t))$ we have $\gamma'=u'\Psi_u+v'\Psi_v$. Thus 
$$\langle{\bf n}(t),V\rangle=-\frac{\mbox{det}(\gamma,\gamma',V)}{|\gamma'|}=-\frac{v'}{|\gamma'|}\mbox{det}(\Psi,\Psi_v,\Psi_u)=-\frac{v'\cosh(u)}{|\gamma'|}.$$
This expression together \eqref{c51} gives the result.
\end{proof}

Finally, notice that the surface of revolution $S_\gamma^{xy}$ of spherical type obtained by rotating an intrinsic catenary $\gamma$ about $L$ is not minimal. Its mean curvature is given from \eqref{h1}, which together \eqref{c51} gives for $\lambda=0$,
\begin{equation*}
\begin{split}
H&=-\frac{\delta}{2|\gamma'|\sinh(u)}\left(\cosh(u)^2+\kappa\sinh(u)|\gamma'|\right)\\
&=-\frac{\delta\cosh(u)}{2|\gamma'|\sinh(u)}\left(\cosh(u)-\frac{v'\sinh(u)}{u}\right).
\end{split}
\end{equation*}
\section*{Acknowledgements}
Rafael  L\'opez  is a member of the Institute of Mathematics  of the University of Granada. This work  has been partially supported by  the Projects  I+D+i PID2020-117868GB-I00, supported by MCIN/ AEI/10.13039/501100011033/,  A-FQM-139-UGR18 and P18-FR-4049. 

\section*{Data Availability Statement} 

My manuscript has no associate data. No funding was received for conducting this study. The authors have no competing interests to declare that are relevant to the content of this article.

\end{document}